\documentclass[11pt,fleqn]{article}

\usepackage{amsmath,amssymb,amsthm}
\usepackage{geometry}
\usepackage[pdfpagemode=UseNone,pdfstartview=FitH,hypertex]{hyperref}

\newcommand{\bdry}[1]{\partial #1}
\newcommand{\dint}{\ds{\int}}
\newcommand{\dist}[2]{\text{dist}\, (#1,#2)}
\newcommand{\ds}[1]{\displaystyle #1}
\newcommand{\id}[1]{id_{#1}}
\newcommand{\incl}{\subset}
\newcommand{\norm}[2][]{\left\|#2\right\|_{#1}}
\renewcommand{\O}{\text{O}}
\renewcommand{\o}{\text{o}}
\newcommand{\PS}[1]{$(\text{PS})_{#1}$}
\newcommand{\QED}{\mbox{\qedhere}}
\newcommand{\restr}[2]{\left.#1\right|_{#2}}
\newcommand{\seq}[1]{\left(#1\right)}
\newcommand{\set}[1]{\left\{#1\right\}}

\newcommand{\N}{\mathbb N}
\newcommand{\R}{\mathbb R}
\newcommand{\RP}{\R \text{P}}
\newcommand{\Z}{\mathbb Z}

\newcommand{\A}{{\cal A}}
\newcommand{\F}{{\cal F}}

\def\cprime{$'$}

\DeclareMathOperator{\divg}{div}

\newenvironment{properties}[1]{\begin{enumerate}
\renewcommand{\theenumi}{$(#1_\arabic{enumi})$}
\renewcommand{\labelenumi}{$(#1_\arabic{enumi})$}}{\end{enumerate}}

\newtheorem{lemma}{Lemma}[section]
\newtheorem{proposition}[lemma]{Proposition}
\newtheorem{theorem}[lemma]{Theorem}

\theoremstyle{remark}
\newtheorem{example}[lemma]{Example}
\newtheorem{remark}[lemma]{Remark}

\numberwithin{equation}{section}

\title{\bf $(N,q)$-Laplacian problems with critical Trudinger-Moser nonlinearities}
\author{\bf Yang Yang\thanks{This work was completed while the first-named author was visiting the Department of Mathematical Sciences at the Florida Institute of Technology, and she is grateful for the kind hospitality of the department. Project supported by NSFC-Tian Yuan Special Foundation (No. 11226116), Natural Science Foundation of Jiangsu Province of China for Young Scholars (No. BK2012109), China Scholarship Council (No. 201208320435), and NSFC (Nos. 11501252, 11571176).}\\
School of Science\\
Jiangnan University\\
Wuxi, 214122, China\\
[\bigskipamount]
\bf Kanishka Perera\\
Department of Mathematical Sciences\\
Florida Institute of Technology\\
Melbourne, FL 32901, USA}
\date{}

\begin{document}

\maketitle

\begin{abstract}
We obtain nontrivial solutions of a $(N,q)$-Laplacian problem with a critical \linebreak Trudinger-Moser nonlinearity in a bounded domain. In addition to the usual difficulty of the loss of compactness associated with problems involving critical nonlinearities, this problem lacks a direct sum decomposition suitable for applying the classical linking theorem. We show that every Palais-Smale sequence at a level below a certain energy threshold admits a subsequence that converges weakly to a nontrivial critical point of the variational functional. Then we prove an abstract critical point theorem based on a cohomological index and use it to construct a minimax level below this threshold.
\end{abstract}

\section{Introduction and main results}

The $(p,q)$-Laplacian operator
\[
\Delta_p\, u + \Delta_q\, u = \divg \left[\left(|\nabla u|^{p-2} + |\nabla u|^{q-2}\right) \nabla u\right]
\]
appears in a wide range of applications that include biophysics \cite{MR527914}, plasma physics \cite{MR901723}, reaction-diffusion equations \cite{MR518461,MR2126276}, and models of elementary particles \cite{MR0174304,MR1626832,MR1785469}. Consequently, quasilinear elliptic boundary value problems involving this operator have been widely studied in the literature (see, e.g., \cite{MR1929880,MR2536276,MR2834776,MR2988766} and the references therein). In particular, the critical $(p,q)$-Laplacian problem
\[
\left\{\begin{aligned}
- \Delta_p\, u - \Delta_q\, u & = \mu\, |u|^{r-2}\, u + |u|^{p^\ast - 2}\, u && \text{in } \Omega\\[10pt]
u & = 0 && \text{on } \bdry{\Omega},
\end{aligned}\right.
\]
where $\Omega$ is a bounded domain in $\R^N,\, N > p > q > 1$, $\mu > 0$, and $p^\ast = Np/(N - p)$ is the critical Sobolev exponent, has been studied by Li and Zhang \cite{MR2509998} in the case $1 < r < q$ and by Yin and Yang \cite{MR2890966} in the case $p < r < p^\ast$. The borderline case
\begin{equation} \label{11}
\left\{\begin{aligned}
- \Delta_p\, u - \Delta_q\, u & = \mu\, |u|^{q-2}\, u + \lambda\, |u|^{p-2}\, u + |u|^{p^\ast - 2}\, u && \text{in } \Omega\\[10pt]
u & = 0 && \text{on } \bdry{\Omega}
\end{aligned}\right.
\end{equation}
with $\mu \in \R$ and $\lambda > 0$ was recently studied in Candito et al.\! \cite{CaMaPe}.

When $N = p \ge 2$, the critical growth is of exponential type and is governed by the Trudinger-Moser inequality
\begin{equation} \label{1.3}
\sup_{u \in W^{1,N}_0(\Omega),\; \norm[N]{\nabla u} \le 1}\, \int_\Omega e^{\, \alpha_N\, |u|^{N'}} dx < \infty,
\end{equation}
where $W^{1,N}_0(\Omega)$ is the usual Sobolev space with the norm $\norm[N]{\nabla u} = \left(\int_\Omega |\nabla u|^N\, dx\right)^{1/N}$, $\alpha_N = N \omega_{N-1}^{1/(N-1)}$, $\omega_{N-1}$ is the area of the unit sphere in $\R^N$, and $N' = N/(N - 1)$ (see Trudinger \cite{MR0216286} and Moser \cite{MR0301504}). A natural analogue of problem \eqref{11} for this case is
\begin{equation} \label{1}
\left\{\begin{aligned}
- \Delta_N\, u - \Delta_q\, u & = \mu\, |u|^{q-2}\, u + \lambda\, |u|^{N-2}\, u\, e^{\, |u|^{N'}} && \text{in } \Omega\\[10pt]
u & = 0 && \text{on } \bdry{\Omega},
\end{aligned}\right.
\end{equation}
where $\mu \in \R$ and $\lambda > 0$, which is the object of study of the present paper. In addition to the usual difficulty of the lack of compactness associated with problems involving critical nonlinearities, this problem is further complicated by the absence of a direct sum decomposition suitable for applying the linking theorem when $\mu$ is above the second eigenvalue of the eigenvalue problem
\begin{equation} \label{7}
\left\{\begin{aligned}
- \Delta_q\, u & = \mu\, |u|^{q-2}\, u && \text{in } \Omega\\[10pt]
u & = 0 && \text{on } \bdry{\Omega}.
\end{aligned}\right.
\end{equation}
To overcome this difficulty, we will first prove an abstract critical point theorem based on a cohomological index that generalizes the classical linking theorem of Rabinowitz \cite{MR0488128}.

Weak solutions of problem \eqref{1} coincide with critical points of the $C^1$-functional
\begin{equation} \label{2}
\Phi(u) = \int_\Omega \left[\frac{1}{N}\, |\nabla u|^N + \frac{1}{q}\, |\nabla u|^q - \frac{\mu}{q}\, |u|^q - \lambda\, F(u)\right] dx, \quad u \in W^{1,N}_0(\Omega),
\end{equation}
where
\begin{equation} \label{2.1}
F(t) = \int_0^t |s|^{N-2}\, s\, e^{\, |s|^{N'}} ds = \int_0^{|t|} s^{N-1}\, e^{\, s^{N'}} ds.
\end{equation}
Recall that $\Phi$ satisfies the Palais-Smale compactness condition at the level $c \in \R$, or \PS{c} for short, if every sequence $\seq{u_j} \subset W^{1,N}_0(\Omega)$ such that $\Phi(u_j) \to c$ and $\Phi'(u_j) \to 0$, called a \PS{c} sequence, has a convergent subsequence. Our existence results will be based on the following proposition.

\begin{proposition} \label{Proposition 1}
If $c < \alpha_N^{N-1}/N$ and $c \ne 0$, then every {\em \PS{c}} sequence has a subsequence that converges weakly to a nontrivial critical point of $\Phi$.
\end{proposition}

Let
\begin{equation} \label{6}
\mu_1 = \inf_{u \in W^{1,q}_0(\Omega) \setminus \set{0}}\, \frac{\norm[q]{\nabla u}^q}{\norm[q]{u}^q} > 0
\end{equation}
be the first eigenvalue of the eigenvalue problem \eqref{7}. First we seek a nonnegative nontrivial solution of problem \eqref{1} when $\mu \le \mu_1$. We assume that $q > N/2$, so that $N < q^\ast = Nq/(N - q)$. Our first main result is the following theorem.

\begin{theorem} \label{Theorem 5}
Assume that $N/2 < q < N$. If $\mu < \mu_1$, then there exists $\lambda^\ast(\mu) > 0$ such that problem \eqref{1} has a nonnegative nontrivial solution for all $\lambda \ge \lambda^\ast(\mu)$.
\end{theorem}

Let $u^\pm(x) = \max \set{\pm u(x),0}$ be the positive and negative parts of $u$, respectively, and set
\[
\Phi^+(u) = \int_\Omega \left[\frac{1}{N}\, |\nabla u|^N + \frac{1}{q}\, |\nabla u|^q - \frac{\mu}{q}\, (u^+)^q - \lambda\, F(u^+)\right] dx, \quad u \in W^{1,N}_0(\Omega).
\]
If $u$ is a critical point of $\Phi^+$, then
\[
{\Phi^+}'(u)\, u^- = \int_\Omega \left(|\nabla u^-|^N + |\nabla u^-|^q\right) dx = 0
\]
and hence $u^- = 0$, so $u = u^+$ is a critical point of $\Phi$ and therefore a nonnegative solution of problem \eqref{1}. Proof of Theorem \ref{Theorem 5} will be based on constructing a minimax level of mountain pass type for $\Phi^+$ below the threshold level given in Proposition \ref{Proposition 1}.

Next we seek a (possibly nodal) nontrivial solution of problem \eqref{1} when $\mu \ge \mu_1$. We have the following theorem.

\begin{theorem} \label{Theorem 6}
Assume that $N/2 < q < N$. If $\mu \ge \mu_1$, then there exists $\lambda^\ast(\mu) > 0$ such that problem \eqref{1} has a nontrivial solution for all $\lambda \ge \lambda^\ast(\mu)$.
\end{theorem}

This extension of Theorem \ref{Theorem 5} is nontrivial. Indeed, the functional $\Phi$ does not have the mountain pass geometry when $\mu \ge \mu_1$ since the origin is no longer a local minimizer, and a linking type argument is needed. However, the classical linking theorem cannot be used since the nonlinear operator $- \Delta_q$ does not have linear eigenspaces. We will use a more general construction based on sublevel sets as in Perera and Szulkin \cite{MR2153141} (see also Perera et al.\! \cite[Proposition 3.23]{MR2640827}). Moreover, the standard sequence of eigenvalues of $- \Delta_q$ based on the genus does not give enough information about the structure of the sublevel sets to carry out this linking construction. Therefore we will use a different sequence of eigenvalues introduced in Perera \cite{MR1998432} that is based on a cohomological index.

The $\Z_2$-cohomological index of Fadell and Rabinowitz \cite{MR57:17677} is defined as follows. Let $W$ be a Banach space and let $\A$ denote the class of symmetric subsets of $W \setminus \set{0}$. For $A \in \A$, let $\overline{A} = A/\Z_2$ be the quotient space of $A$ with each $u$ and $-u$ identified, let $f : \overline{A} \to \RP^\infty$ be the classifying map of $\overline{A}$, and let $f^\ast : H^\ast(\RP^\infty) \to H^\ast(\overline{A})$ be the induced homomorphism of the Alexander-Spanier cohomology rings. The cohomological index of $A$ is defined by
\[
i(A) = \begin{cases}
\sup \set{m \ge 1 : f^\ast(\omega^{m-1}) \ne 0}, & A \ne \emptyset\\[5pt]
0, & A = \emptyset,
\end{cases}
\]
where $\omega \in H^1(\RP^\infty)$ is the generator of the polynomial ring $H^\ast(\RP^\infty) = \Z_2[\omega]$.

\begin{example}
The classifying map of the unit sphere $S^{m-1}$ in $\R^m,\, m \ge 1$ is the inclusion $\RP^{m-1} \incl \RP^\infty$, which induces isomorphisms on the cohomology groups $H^q$ for $q \le m - 1$, so $i(S^{m-1}) = m$.
\end{example}

The following proposition summarizes the basic properties of this index.

\begin{proposition}[Fadell-Rabinowitz \cite{MR57:17677}] \label{Proposition 7}
The index $i : \A \to \N \cup \set{0,\infty}$ has the following properties:
\begin{properties}{i}
\item Definiteness: $i(A) = 0$ if and only if $A = \emptyset$;
\item \label{i2} Monotonicity: If there is an odd continuous map from $A$ to $B$ (in particular, if $A \subset B$), then $i(A) \le i(B)$. Thus, equality holds when the map is an odd homeomorphism;
\item Dimension: $i(A) \le \dim W$;
\item Continuity: If $A$ is closed, then there is a closed neighborhood $N \in \A$ of $A$ such that $i(N) = i(A)$. When $A$ is compact, $N$ may be chosen to be a $\delta$-neighborhood $N_\delta(A) = \set{u \in W : \dist{u}{A} \le \delta}$;
\item Subadditivity: If $A$ and $B$ are closed, then $i(A \cup B) \le i(A) + i(B)$;
\item \label{i6} Stability: If $SA$ is the suspension of $A \ne \emptyset$, obtained as the quotient space of $A \times [-1,1]$ with $A \times \set{1}$ and $A \times \set{-1}$ collapsed to different points, then $i(SA) = i(A) + 1$;
\item \label{i7} Piercing property: If $A$, $A_0$ and $A_1$ are closed, and $\varphi : A \times [0,1] \to A_0 \cup A_1$ is a continuous map such that $\varphi(-u,t) = - \varphi(u,t)$ for all $(u,t) \in A \times [0,1]$, $\varphi(A \times [0,1])$ is closed, $\varphi(A \times \set{0}) \subset A_0$ and $\varphi(A \times \set{1}) \subset A_1$, then $i(\varphi(A \times [0,1]) \cap A_0 \cap A_1) \ge i(A)$;
\item Neighborhood of zero: If $U$ is a bounded closed symmetric neighborhood of $0$, then $i(\bdry{U}) = \dim W$.
\end{properties}
\end{proposition}

The Dirichlet spectrum of $- \Delta_q$ in $\Omega$ consists of those $\mu \in \R$ for which problem \eqref{7} has a nontrivial solution. Although a complete description of the spectrum is not yet known when $N \ge 2$, we can define an increasing and unbounded sequence of eigenvalues via a suitable minimax scheme. The standard scheme based on the genus does not give the index information necessary to prove Theorem \ref{Theorem 6}, so we will use the following scheme based on the cohomological index as in Perera \cite{MR1998432}. Let
\[
\Psi(u) = \frac{1}{\dint_\Omega |u|^q\, dx}, \quad u \in S_q = \set{u \in W^{1,q}_0(\Omega) : \int_\Omega |\nabla u|^q\, dx = 1}.
\]
Then eigenvalues of problem \eqref{7} on $S_q$ coincide with critical values of $\Psi$. We use the standard notation
\[
\Psi^a = \set{u \in S_q : \Psi(u) \le a}, \quad \Psi_a = \set{u \in S_q : \Psi(u) \ge a}, \quad a \in \R
\]
for the sublevel sets and superlevel sets, respectively. Let $\F$ denote the class of symmetric subsets of $S_q$ and set
\[
\mu_k := \inf_{M \in \F,\; i(M) \ge k}\, \sup_{u \in M}\, \Psi(u), \quad k \in \N.
\]
Then $0 < \mu_1 < \mu_2 \le \mu_3 \le \cdots \to + \infty$ is a sequence of eigenvalues of problem \eqref{7} and
\begin{equation} \label{15}
\mu_k < \mu_{k+1} \implies i(\Psi^{\mu_k}) = i(S_q \setminus \Psi_{\mu_{k+1}}) = k
\end{equation}
(see Perera et al.\! \cite[Propositions 3.52 and 3.53]{MR2640827}).

Proof of Theorem \ref{Theorem 6} will make essential use of \eqref{15} and will be based on the following abstract critical point theorem, which is of independent interest. Let $W$ be a Banach space, let
\[
S = \set{u \in W : \norm{u} = 1}
\]
be the unit sphere in $W$, and let
\[
\pi : W \setminus \set{0} \to S, \quad u \mapsto \frac{u}{\norm{u}}
\]
be the radial projection onto $S$.

\begin{theorem} \label{Theorem 8}
Let $\Phi$ be a $C^1$-functional on $W$ and let $A_0,\, B_0$ be disjoint nonempty closed symmetric subsets of $S$ such that
\begin{equation} \label{16}
i(A_0) = i(S \setminus B_0) < \infty.
\end{equation}
Assume that there exist $R > r > 0$ and $v \in S \setminus A_0$ such that
\[
\sup \Phi(A) \le \inf \Phi(B), \qquad \sup \Phi(X) < \infty,
\]
where
\begin{gather*}
A = \set{tu : u \in A_0,\, 0 \le t \le R} \cup \set{R\, \pi((1 - t)\, u + tv) : u \in A_0,\, 0 \le t \le 1},\\[5pt]
B = \set{ru : u \in B_0},\\[5pt]
X = \set{tu : u \in A,\, \norm{u} = R,\, 0 \le t \le 1}.
\end{gather*}
Let $\Gamma = \set{\gamma \in C(X,W) : \gamma(X) \text{ is closed and} \restr{\gamma}{A} = \id{A}}$ and set
\[
c := \inf_{\gamma \in \Gamma}\, \sup_{u \in \gamma(X)}\, \Phi(u).
\]
Then
\[
\inf \Phi(B) \le c \le \sup \Phi(X)
\]
and $\Phi$ has a {\em \PS{c}} sequence.
\end{theorem}

\begin{remark}
Theorem \ref{Theorem 8}, which does not require a direct sum decomposition, generalizes the linking theorem of Rabinowitz \cite{MR0488128}.
\end{remark}

\section{Preliminaries}

In this preliminary section we prove Proposition \ref{Proposition 1} and Theorem \ref{Theorem 8}. First we obtain some estimates for the primitive $F$ defined in \eqref{2.1}.

\begin{lemma}
For all $t \in \R$,
\begin{gather}
\label{9} F(t) - \frac{|t|^N}{2N}\, e^{\, |t|^{N'}} \le C;\\[5pt]
\label{10} F(t) \le |t|^{N-1}\, e^{\, |t|^{N'}} + C,
\end{gather}
where $C$ denotes a generic positive constant.
\end{lemma}

\begin{proof}
Integrating by parts gives
\[
F(t) = \frac{|t|^N}{N}\, e^{\, |t|^{N'}} - \frac{N'}{N} \int_0^{|t|} s^{N+N'-1}\, e^{\, s^{N'}} ds.
\]
For $|t| \ge (N/N')^{1/N'}$, the last term is greater than or equal to
\[
\frac{N'}{N} \int_{(N/N')^{1/N'}}^{|t|} s^{N+N'-1}\, e^{\, s^{N'}} ds \ge \int_{(N/N')^{1/N'}}^{|t|} s^{N-1}\, e^{\, s^{N'}} ds = F(t) - F((N/N')^{1/N'})
\]
and hence
\[
2F(t) - \frac{|t|^N}{N}\, e^{\, |t|^{N'}} \le F((N/N')^{1/N'}).
\]
Since $F$ is bounded on bounded sets, \eqref{9} follows. As for \eqref{10},
\[
F(t) = \frac{|t|^{N-N'}}{N'}\, e^{\, |t|^{N'}} - \frac{N - N'}{N'} \int_0^{|t|} s^{N-N'-1}\, e^{\, s^{N'}} ds \le |t|^{N-1}\, e^{\, |t|^{N'}}
\]
for $|t| \ge 1/(N')^{1/(N'-1)}$.
\end{proof}

\begin{proof}[Proof of Proposition \ref{Proposition 1}]
Let $0 \ne c < \alpha_N^{N-1}/N$ and let $\seq{u_j}$ be a \PS{c} sequence. Then
\begin{equation} \label{4}
\Phi(u_j) = \int_\Omega \left[\frac{1}{N}\, |\nabla u_j|^N + \frac{1}{q}\, |\nabla u_j|^q - \frac{\mu}{q}\, |u_j|^q - \lambda\, F(u_j)\right] dx = c + \o(1)
\end{equation}
and
\begin{multline} \label{8}
\Phi'(u_j)\, v = \int_\Omega \bigg[\Big(|\nabla u_j|^{N-2} + |\nabla u_j|^{q-2}\Big)\, \nabla u_j \cdot \nabla v\\[5pt]
- \left(\mu\, |u_j|^{q-2} + \lambda\, |u_j|^{N-2}\, e^{\, |u_j|^{N'}}\right) u_j\, v\bigg]\, dx = \o(1) \norm[N]{\nabla v} \quad \forall v \in W^{1,N}_0(\Omega),
\end{multline}
in particular,
\begin{equation} \label{5}
\Phi'(u_j)\, u_j = \int_\Omega \left(|\nabla u_j|^N + |\nabla u_j|^q - \mu\, |u_j|^q - \lambda\, |u_j|^N\, e^{\, |u_j|^{N'}}\right) dx = \o(1) \norm[N]{\nabla u_j}.
\end{equation}
By \eqref{4}, \eqref{5}, and \eqref{9},
\[
\int_\Omega \left[\frac{1}{2N}\, |\nabla u_j|^N + \left(\frac{1}{q} - \frac{1}{2N}\right)\! \left(|\nabla u_j|^q - \mu\, |u_j|^q\right)\right] dx = \o(1) \norm[N]{\nabla u_j} + \O(1).
\]
Since $N > q > 1$, it follows from this that the sequence $\seq{u_j}$ is bounded in $W^{1,N}_0(\Omega)$. So a renamed subsequence converges to some $u$ weakly in $W^{1,N}_0(\Omega)$, strongly in $L^s(\Omega)$ for all $1 \le s < \infty$, and a.e.\! in $\Omega$. Since $\int_\Omega |u_j|^N\, e^{\, |u_j|^{N'}} dx$ is bounded by \eqref{5}, then for any $v \in C^\infty_0(\Omega)$,
\[
\int_\Omega |u_j|^{N-2}\, u_j\, e^{\, |u_j|^{N'}} v\, dx \to \int_\Omega |u|^{N-2}\, u\, e^{\, |u|^{N'}} v\, dx
\]
by de Figueiredo et al.\! \cite[Lemma 2.1]{MR1386960} and hence passing to the limit in \eqref{8} gives
\[
\int_\Omega \bigg[\Big(|\nabla u|^{N-2} + |\nabla u|^{q-2}\Big)\, \nabla u \cdot \nabla v - \left(\mu\, |u|^{q-2} + \lambda\, |u|^{N-2}\, e^{\, |u|^{N'}}\right) uv\bigg]\, dx = 0.
\]
This then holds for all $v \in W^{1,N}_0(\Omega)$ by density, so $u$ is a critical point of $\Phi$.

Suppose $u = 0$. Then
\[
\int_\Omega |u_j|^{N-1}\, e^{\, |u_j|^{N'}} dx \to 0
\]
by de Figueiredo et al.\! \cite[Lemma 2.1]{MR1386960} as above and hence
\[
\int_\Omega F(u_j)\, dx \to 0
\]
by \eqref{10} and the dominated convergence theorem, so \eqref{4} gives
\[
\int_\Omega \left[\frac{1}{N}\, |\nabla u_j|^N + \frac{1}{q}\, |\nabla u_j|^q\right] dx \to c.
\]
Since $c < \alpha_N^{N-1}/N$, then $\varlimsup\, \norm[N]{\nabla u_j} < \alpha_N^{1/N'}$, so there exists $\beta > 1/\alpha_N^{1/N'}$ such that $\beta \norm[N]{\nabla u_j} \le 1$ for all sufficiently large $j$. For $1 < \gamma < \infty$ given by $1/\alpha_N \beta^{N'} + 1/\gamma = 1$, then
\[
\int_\Omega |u_j|^N\, e^{\, |u_j|^{N'}} dx \le \left(\int_\Omega |u_j|^{\gamma N}\, dx\right)^{1/\gamma} \left(\int_\Omega e^{\, \alpha_N\, |\beta u_j|^{N'}} dx\right)^{1/\alpha_N \beta^{N'}} \to 0
\]
since $u_j \to 0$ in $L^{\gamma N}(\Omega)$ and the last integral is bounded by \eqref{1.3}. Then
\[
\int_\Omega \left(|\nabla u_j|^N + |\nabla u_j|^q\right) dx \to 0
\]
by \eqref{5} and hence $u_j \to 0$ in $W^{1,N}_0(\Omega)$, so $\Phi(u_j) \to 0$, contradicting $c \ne 0$.
\end{proof}

\begin{proof}[Proof of Theorem \ref{Theorem 8}]
First we show that $A$ (homotopically) links $B$ with respect to $X$ in the sense that
\begin{equation} \label{17}
\gamma(X) \cap B \ne \emptyset \quad \forall \gamma \in \Gamma.
\end{equation}
If \eqref{17} does not hold, then there is a map $\gamma \in C(X,W \setminus B)$ such that $\gamma(X)$ is closed and $\restr{\gamma}{A} = \id{A}$. Let
\[
\widetilde{A} = \set{R\, \pi((1 - |t|)\, u + tv) : u \in A_0,\, -1 \le t \le 1}
\]
and note that $\widetilde{A}$ is closed since $A_0$ is closed (here $(1 - |t|)\, u + tv \ne 0$ since $v$ is not in the symmetric set $A_0$). Since
\[
SA_0 \to \widetilde{A}, \quad (u,t) \mapsto R\, \pi((1 - |t|)\, u + tv)
\]
is an odd continuous map,
\begin{equation} \label{18}
i(\widetilde{A}) \ge i(SA_0) = i(A_0) + 1
\end{equation}
by \ref{i2} and \ref{i6} of Proposition \ref{Proposition 7}. Consider the map
\[
\varphi : \widetilde{A} \times [0,1] \to W \setminus B, \quad \varphi(u,t) = \begin{cases}
\gamma(tu), & u \in \widetilde{A} \cap A\\[5pt]
- \gamma(-tu), & u \in \widetilde{A} \setminus A,
\end{cases}
\]
which is continuous since $\gamma$ is the identity on the symmetric set $\set{tu : u \in A_0,\, 0 \le t \le R}$. We have $\varphi(-u,t) = - \varphi(u,t)$ for all $(u,t) \in \widetilde{A} \times [0,1]$, $\varphi(\widetilde{A} \times [0,1]) = \gamma(X) \cup (- \gamma(X))$ is closed, and $\varphi(\widetilde{A} \times \set{0}) = \set{0}$ and $\varphi(\widetilde{A} \times \set{1}) = \widetilde{A}$ since $\restr{\gamma}{A} = \id{A}$. Applying \ref{i7} with $\widetilde{A}_0 = \set{u \in W : \norm{u} \le r}$ and $\widetilde{A}_1 = \set{u \in W : \norm{u} \ge r}$ gives
\begin{equation} \label{19}
i(\widetilde{A}) \le i(\varphi(\widetilde{A} \times [0,1]) \cap \widetilde{A}_0 \cap \widetilde{A}_1) \le i((W \setminus B) \cap S_r) = i(S_r \setminus B) = i(S \setminus B_0),
\end{equation}
where $S_r = \set{u \in W : \norm{u} = r}$. By \eqref{18} and \eqref{19}, $i(A_0) < i(S \setminus B_0)$, contradicting \eqref{16}. Hence \eqref{17} holds.

It follows from \eqref{17} that $c \ge \inf \Phi(B)$, and $c \le \sup \Phi(X)$ since $\id{X} \in \Gamma$. By a standard argument, $\Phi$ has a \PS{c} sequence (see, e.g., Ghoussoub \cite{MR1251958}).
\end{proof}

\begin{remark}
The linking construction in the above proof was used in Perera and Szulkin \cite{MR2153141} to obtain nontrivial solutions of $p$-Laplacian problems with nonlinearities that interact with the spectrum. A similar construction based on the notion of cohomological linking was given in Degiovanni and Lancelotti \cite{MR2371112}. See also Perera et al.\! \cite[Proposition 3.23]{MR2640827}.
\end{remark}

\section{Proof of Theorem \ref{Theorem 5}}

Fix $u_0 > 0$ in $W^{1,N}_0(\Omega)$ such that $\norm[N+N']{u_0} = 1$. Since $e^t \ge t$ for all $t \ge 0$,
\[
F(t) \ge \frac{|t|^{N+N'}}{N + N'} \quad \forall t \in \R
\]
and hence
\begin{equation} \label{12}
\Phi^+(t u_0) \le \int_\Omega \left[\frac{t^N}{N}\, |\nabla u_0|^N + \frac{t^q}{q}\, |\nabla u_0|^q - \frac{\mu\, t^q}{q}\, u_0^q\right] dx - \frac{\lambda\, t^{N+N'}}{N + N'} \to - \infty
\end{equation}
as $t \to + \infty$. Take $t_0 > 0$ so large that $\Phi^+(t_0 u_0) \le 0$, let
\[
\Gamma = \set{\gamma \in C([0,1],W^{1,N}_0(\Omega)) : \gamma(0) = 0,\, \gamma(1) = t_0 u_0}
\]
be the class of paths joining $0$ and $t_0 u_0$, and set
\[
c := \inf_{\gamma \in \Gamma}\, \max_{u \in \gamma([0,1])}\, \Phi^+(u).
\]

\begin{lemma} \label{Lemma 2}
If $0 < c < \alpha_N^{N-1}/N$, then problem \eqref{1} has a nonnegative nontrivial solution.
\end{lemma}

\begin{proof}
By the mountain pass theorem, $\Phi^+$ has a \PS{c} sequence $\seq{u_j}$. An argument similar to that in the proof of Proposition \ref{Proposition 1} shows that a subsequence of $\seq{u_j}$ converges weakly to a nontrivial critical point of $\Phi^+$.
\end{proof}

We have the following upper bound for $c$.

\begin{lemma} \label{Lemma 3}
Let $\widetilde{\lambda} = \lambda/2$. Then
\begin{multline*}
c \le \frac{1}{N^2\, \widetilde{\lambda}^{N-1}} \left(\int_\Omega |\nabla u_0|^N\, dx\right)^N\\[5pt]
+ \frac{N + N' - q}{q\, (N + N')\, \widetilde{\lambda}^{q/(N+N'-q)}} \left[\int_\Omega \big(|\nabla u_0|^q - \mu\, u_0^q\big)\, dx\right]^{(N+N')/(N+N'-q)}.
\end{multline*}
\end{lemma}

\begin{proof}
Since $\gamma(s) = s t_0 u_0$ is a path in $\Gamma$ and \eqref{12} holds,
\begin{multline*}
c \le \max_{s \in [0,1]}\, \Phi^+(s t_0 u_0) \le \max_{t \ge 0}\, \Phi^+(t u_0) \le \max_{t \ge 0}\, \left[\frac{t^N}{N} \int_\Omega |\nabla u_0|^N\, dx - \frac{\widetilde{\lambda}\, t^{N+N'}}{N + N'}\right]\\[5pt]
+ \max_{t \ge 0}\, \left[\frac{t^q}{q} \int_\Omega \big(|\nabla u_0|^q - \mu\, u_0^q\big)\, dx - \frac{\widetilde{\lambda}\, t^{N+N'}}{N + N'}\right]. \QED
\end{multline*}
\end{proof}

We are now ready to prove Theorem \ref{Theorem 5}.

\begin{proof}[Proof of Theorem \ref{Theorem 5}]
We apply Lemma \ref{Lemma 2}. By \eqref{6},
\[
\Phi^+(u) \ge \frac{1}{N} \norm[N]{\nabla u}^N + \frac{1}{q} \left(1 - \frac{\mu^+}{\mu_1}\right) \norm[q]{\nabla u}^q - \lambda \int_\Omega |u|^N\, e^{\, |u|^{N'}} dx \quad \forall u \in W^{1,N}_0(\Omega),
\]
where $\mu^+ = \max \set{\mu,0} < \mu_1$, since $F(u^+) \le F(|u|) \le |u|^N\, e^{\, |u|^{N'}}$. Since $q > N/2$, $N < q^\ast = Nq/(N - q)$. For $1/\alpha_N^{1/N'} < \beta < \infty$ given by $1/\alpha_N \beta^{N'} + N/q^\ast = 1$,
\[
\int_\Omega |u|^N\, e^{\, |u|^{N'}} dx \le \left(\int_\Omega e^{\, \alpha_N\, |\beta u|^{N'}} dx\right)^{1/\alpha_N \beta^{N'}} \norm[q^\ast]{u}^N
\]
and the last integral is bounded for all $u \in W^{1,N}_0(\Omega)$ with $\norm[N]{\nabla u} \le 1/\beta$ by \eqref{1.3}. Since $q < N$, $W^{1,N}_0(\Omega) \hookrightarrow W^{1,q}_0(\Omega) \hookrightarrow L^{q^\ast}(\Omega)$ and it follows that $0$ is a strict local minimizer of $\Phi^+$. So $c > 0$. It is clear from Lemma \ref{Lemma 3} that $c < \alpha_N^{N-1}/N$ if $\lambda > 0$ is sufficiently large.
\end{proof}

\section{Proof of Theorem \ref{Theorem 6}}

\begin{proof}[Proof of Theorem \ref{Theorem 6}]
Since $q < N$, $W^{1,N}_0(\Omega) \hookrightarrow W^{1,q}_0(\Omega)$. Let $S_N$ and $S_q$ denote the unit spheres of $W^{1,N}_0(\Omega)$ and $W^{1,q}_0(\Omega)$, respectively, and let
\[
\pi_N(u) = \frac{u}{\norm[N]{\nabla u}}, \quad u \in W^{1,N}_0(\Omega) \setminus \set{0}, \qquad \pi_q(u) = \frac{u}{\norm[q]{\nabla u}}, \quad u \in W^{1,q}_0(\Omega) \setminus \set{0}
\]
be the radial projections onto $S_N$ and $S_q$, respectively. Since $\mu \ge \mu_1$, $\mu_k \le \mu < \mu_{k+1}$ for some $k \ge 1$. Then
\begin{equation} \label{20}
i(\pi_q^{-1}(\Psi^{\mu_k})) = i(\pi_q^{-1}(S_q \setminus \Psi_{\mu_{k+1}})) = k
\end{equation}
by \eqref{15}. Set $M = \big\{u \in W^{1,q}_0(\Omega) : \norm[q]{u} = 1\big\}$. By \cite[Theorem 2.3]{MR2514055}, the set $\pi_q^{-1}(\Psi^{\mu_k}) \cup \set{0}$ contains a symmetric cone $C$ such that $C \cap M$ is compact in $C^1(\Omega)$ and
\begin{equation} \label{21}
i(C \setminus \set{0}) = k.
\end{equation}
Since $W^{1,N}_0(\Omega)$ is a dense linear subspace of $W^{1,q}_0(\Omega)$, the inclusion $\pi_q^{-1}(S_q \setminus \Psi_{\mu_{k+1}}) \cap W^{1,N}_0(\Omega) \incl \pi_q^{-1}(S_q \setminus \Psi_{\mu_{k+1}})$ is a homotopy equivalence by Palais \cite[Theorem 17]{MR0189028}, so
\begin{equation} \label{22}
i(\pi_q^{-1}(S_q \setminus \Psi_{\mu_{k+1}}) \cap W^{1,N}_0(\Omega)) = k
\end{equation}
by \eqref{20}. We apply Theorem \ref{Theorem 8} to our functional $\Phi$ defined in \eqref{2} with
\[
A_0 = \pi_N(C \setminus \set{0}) = \pi_N(C \cap M), \qquad B_0 = S_N \setminus (\pi_q^{-1}(S_q \setminus \Psi_{\mu_{k+1}}) \cap W^{1,N}_0(\Omega)),
\]
noting that $A_0$ is compact since $C \cap M$ is compact and $\pi_N$ is continuous. We have
\[
i(A_0) = i(C \setminus \set{0}) = k
\]
by \eqref{21}, and
\[
i(S_N \setminus B_0) = i(\pi_q^{-1}(S_q \setminus \Psi_{\mu_{k+1}}) \cap W^{1,N}_0(\Omega)) = k
\]
by \eqref{22}, so \eqref{16} holds.

For $u \in S_N$ and $t \ge 0$, since $F(u) \ge |u|^N/N$,
\begin{equation} \label{23}
\Phi(tu) \le \frac{t^q}{q} \int_\Omega \big(|\nabla u|^q - \mu\, |u|^q\big)\, dx - \frac{\widetilde{\lambda}\, t^N}{N} \int_\Omega |u|^N\, dx - \frac{t^N}{N} \left(\widetilde{\lambda} \int_\Omega |u|^N\, dx - 1\right),
\end{equation}
where $\widetilde{\lambda} = \lambda/2$. Pick any $v \in S_N \setminus A_0$. Since $A_0$ is compact, so is the set
\[
X_0 = \set{\pi_N((1 - t)\, u + tv) : u \in A_0,\, 0 \le t \le 1}
\]
and hence
\[
\alpha = \inf_{u \in X_0}\, \int_\Omega |u|^N\, dx > 0, \qquad \beta = \sup_{u \in X_0}\, \int_\Omega \big(|\nabla u|^q - \mu\, |u|^q\big)\, dx < \infty.
\]
Let $\lambda \ge 2/\alpha$, so that $\widetilde{\lambda} \alpha \ge 1$. Then for $u \in A_0 \subset X_0$ and $t \ge 0$, \eqref{23} gives
\begin{equation} \label{24}
\Phi(tu) \le - (\mu - \mu_k)\, \frac{t^q}{q} \int_\Omega |u|^q\, dx \le 0
\end{equation}
since $\mu \ge \mu_k$. For $u \in X_0$ and $t \ge 0$, \eqref{23} gives
\begin{equation} \label{25}
\Phi(tu) \le \frac{\beta\, t^q}{q} - \frac{\widetilde{\lambda} \alpha\, t^N}{N} \le \left(\dfrac{1}{q} - \dfrac{1}{N}\right) \dfrac{(\beta^+)^{N/(N-q)}}{(\widetilde{\lambda} \alpha)^{q/(N-q)}},
\end{equation}
where $\beta^+ = \max \set{\beta,0}$. Fix $\lambda$ so large that the last expression is $< \alpha_N^{N-1}/N$, take a positive $R \ge (N \beta^+/q\, \widetilde{\lambda} \alpha)^{1/(N-q)}$, and let $A$ and $X$ be as in Theorem \ref{Theorem 8}. Then it follows from \eqref{24} and \eqref{25} that
\[
\sup \Phi(A) \le 0, \qquad \sup \Phi(X) < \frac{\alpha_N^{N-1}}{N}.
\]
For $u \in \pi_N^{-1}(B_0)$,
\[
\Phi(u) \ge \frac{1}{N} \norm[N]{\nabla u}^N + \frac{1}{q} \left(1 - \frac{\mu}{\mu_{k+1}}\right) \norm[q]{\nabla u}^q - \lambda \int_\Omega |u|^N\, e^{\, |u|^{N'}} dx
\]
since $F(u) \le |u|^N\, e^{\, |u|^{N'}}$. Since $q > N/2$, $N < q^\ast = Nq/(N - q)$. For $1/\alpha_N^{1/N'} < \beta < \infty$ given by $1/\alpha_N \beta^{N'} + N/q^\ast = 1$,
\[
\int_\Omega |u|^N\, e^{\, |u|^{N'}} dx \le \left(\int_\Omega e^{\, \alpha_N\, |\beta u|^{N'}} dx\right)^{1/\alpha_N \beta^{N'}} \norm[q^\ast]{u}^N
\]
and the last integral is bounded for all $u \in W^{1,N}_0(\Omega)$ with $\norm[N]{\nabla u} \le 1/\beta$ by \eqref{1.3}. Since $q < N$ and $\mu < \mu_{k+1}$, $W^{1,N}_0(\Omega) \hookrightarrow W^{1,q}_0(\Omega) \hookrightarrow L^{q^\ast}(\Omega)$ and it follows that if $0 < r < R$ is sufficiently small and $B$ is as in Theorem \ref{Theorem 8}, then
\[
\inf \Phi(B) > 0.
\]
Then $0 < c < \alpha_N^{N-1}/N$ and $\Phi$ has a \PS{c} sequence by Theorem \ref{Theorem 8}, a subsequence of which converges weakly to a nontrivial critical point of $\Phi$ by Proposition \ref{Proposition 1}.
\end{proof}

\def\cdprime{$''$}

\end{document}